\documentclass[11pt,a4paper]{amsart}
\usepackage[T1]{fontenc}
\usepackage{lmodern}
\usepackage{microtype}
\usepackage{mathtools,amssymb,mathrsfs}
\usepackage[margin=27mm]{geometry}
\usepackage{enumitem}
\usepackage[colorlinks=true,linkcolor=blue,citecolor=blue,urlcolor=blue,bookmarksnumbered]{hyperref}
\usepackage{bookmark}
\hypersetup{pdftitle={Periodic transport and Oka complements in Cn},pdfauthor={Yun-Heng Du, Bin Guo, Peng-Chao Wang and Song-Yan Xie},pdfsubject={Periodic obstacles, convex tubes, and products of planar compact sets},pdfkeywords={Oka manifold, periodic transport, divergence-free vector field, Fatou-Bieberbach domain, convex tube, Cartesian product}}
\numberwithin{equation}{section}
\newtheorem{theorem}{Theorem}[section]
\newtheorem{proposition}[theorem]{Proposition}
\newtheorem{lemma}[theorem]{Lemma}
\newtheorem{corollary}[theorem]{Corollary}
\theoremstyle{definition}

\theoremstyle{remark}

\newcommand{\C}{\mathbb C}
\newcommand{\R}{\mathbb R}
\newcommand{\Z}{\mathbb Z}

\newcommand{\OO}{\mathscr O}
\newcommand{\Aut}{\operatorname{Aut}}

\newcommand{\id}{\operatorname{id}}
\newcommand{\Int}{\operatorname{int}}
\newcommand{\im}{\operatorname{Im}}
\newcommand{\re}{\operatorname{Re}}

\newcommand{\dist}{\operatorname{dist}}
\newcommand{\Ball}{\mathbb B}
\newcommand{\ii}{\mathrm i}
\newcommand{\dd}{\mathrm d}
\newcommand{\norm}[1]{\left\lVert #1\right\rVert}

\newcommand{\doi}[1]{\href{https://doi.org/#1}{\nolinkurl{#1}}}
\newcommand{\arxiv}[1]{\href{https://arxiv.org/abs/#1}{arXiv:\nolinkurl{#1}}}
\title[Periodic transport and Oka complements]{Periodic transport and Oka complements in \(\C^n\)}
\author[Du]{Yun-Heng Du}
\address[Yun-Heng Du]{Academy of Mathematics and Systems Science, Chinese Academy of Sciences, Beijing 100190, China}
\email{duyunheng@amss.ac.cn}

\author[Guo]{Bin Guo}
\address[Bin Guo]{Academy of Mathematics and Systems Science, Chinese Academy of Sciences, Beijing 100190, China}
\email{guobin181@mails.ucas.ac.cn}

\author[Wang]{Peng-Chao Wang}
\address[Peng-Chao Wang]{Academy of Mathematics and Systems Science, Chinese Academy of Sciences, Beijing 100190, China}
\email{blowa8@gmail.com}

\author[Xie]{Song-Yan Xie}
\address[Song-Yan Xie]{State Key Laboratory of Mathematical Sciences, Academy of Mathematics and Systems Science, Chinese Academy of Sciences, Beijing 100190, China; School of Mathematical Sciences, University of Chinese Academy of Sciences, Beijing 100049, China}
\email{xiesongyan@amss.ac.cn}

\begin{document}
\begin{abstract}
For every \(n\geq2\), we prove that the complement in \(\C^n\) of a periodic closed set with compact holomorphically convex quotient in \((\C^*)^n\) is Oka. The proof combines periodic divergence-free transport with Fatou--Bieberbach basins. Applications include complements of convex tubes \(\R^n+\ii B\), products of closed annuli, and, more generally, products of planar compact sets whose complements have at most one bounded component. In particular, \(\C^2\setminus\R^2\) is Oka.
\end{abstract}
\maketitle

\section{Introduction}\label{sec:intro}

A complex manifold \(Y\) is called \emph{Oka} if it satisfies the following \emph{convex approximation property}: for every integer \(m\geq1\), every compact convex set \(K\subset\C^m\), every open neighborhood \(U\subset\C^m\) of \(K\), every holomorphic map \(f:U\to Y\), and every \(\epsilon>0\), there exists a holomorphic map \(g:\C^m\to Y\) such that \(d_Y(f(z),g(z))<\epsilon\) for all \(z\in K\), where \(d_Y\) is a fixed distance inducing the topology of \(Y\). This property is independent of the choice of \(d_Y\) and is equivalent to the parametric Oka property with approximation and interpolation for maps from Stein manifolds to \(Y\); see \cite[Theorem~0.1]{For06}, \cite[Theorem~5.4.4]{For17}, and \cite[Theorem~1.2 and Section~3]{For23}.

Earlier work on maps into \(\C^2\setminus\R^2\) includes the Mergelyan theorem of Winkelmann \cite{Win98}. Forstneri\v c and Wold asked whether \(\C^n\setminus\R^k\) is Oka for \(n>1\) and \(1\leq k\leq n\) \cite[Problem~1.5]{FW15}. Kusakabe's complement theorem \cite[Theorem~1.6 and Corollary~1.7]{Kus24} proves these cases except for \((n,k)=(2,1),(2,2),(3,3)\). Forstneri\v c and Wold proved that \(\C^2\setminus\R\) is Oka \cite[Proposition~4.9]{FW24}, and Du proved that \(\C^3\setminus S\) is Oka for every closed subset \(S\subset\R^3\) \cite[Theorem~1.3]{Du26}. We prove that \(\C^2\setminus\R^2\) is Oka, settling the remaining case for totally real affine subspaces. This follows from a general construction of holomorphic families of attracting basins in periodic complements, with applications to convex tubes. For \(n\geq3\), the full-plane conclusions give uniform alternative proofs of cases already covered by Kusakabe and Du, whose corresponding results also apply to arbitrary closed subsets of the real planes.

To formulate the general result, fix \(n\geq2\) and consider the covering
\begin{equation}\label{eq:covering}
q:\C^n\longrightarrow X\coloneqq(\C^*)^n,\qquad
q(z)=(e^{\ii z_1},\ldots,e^{\ii z_n}),
\end{equation}
with deck group \(\Gamma=2\pi\Z^n\). A closed set of the form \(E=q^{-1}(C)\), with \(C\subset X\) compact, is unbounded but compact modulo \(\Gamma\). Our principal result concerns such sets when \(C\) is holomorphically convex. All vector norms other than \(\norm{\cdot}_\infty\) are Euclidean, and matrix norms are the induced operator norms.

\begin{theorem}\label{thm:periodic-complement}
For \(n\geq2\), let \(C\subset X=(\C^*)^n\) be a nonempty compact \(\OO(X)\)-convex set, and set \(E=q^{-1}(C)\) and \(Y=\C^n\setminus E\).

Let \(m\geq1\), \(K\subset\C^m\) be compact and convex, \(U_0\subset\C^m\) be an open neighborhood of \(K\), and \(f:U_0\to Y\) be holomorphic. There exist an open neighborhood \(V\subset U_0\) of \(K\) and a holomorphic map
\begin{equation}\label{eq:general-family}
F:V\times\C^n\longrightarrow Y
\end{equation}
such that every \(F_\zeta=F(\zeta,\cdot)\) is biholomorphic onto its image and
\begin{equation}\label{eq:general-normalization}
F(\zeta,0)=f(\zeta),\qquad D_\xi F(\zeta,0)=I_n.
\end{equation}
The images \(\Omega_\zeta=F_\zeta(\C^n)\) are attracting basins of a holomorphic family of automorphisms of \(\C^n\), and
\[
(\zeta,\xi)\longmapsto(\zeta,F(\zeta,\xi))
\]
is a biholomorphism from \(V\times\C^n\) onto the open set \(\{(\zeta,z):\zeta\in V,\ z\in\Omega_\zeta\}\). In particular, \(Y\) is Oka.
\end{theorem}

Here \(I_n\) denotes the \(n\times n\) identity matrix. A Fatou--Bieberbach domain in \(\C^n\) is a proper domain biholomorphic to \(\C^n\). The nonemptiness of \(C\) guarantees that the domains in the theorem are proper; the Oka assertion for \(C=\varnothing\) is immediate. Both \(X\setminus C\) and \(Y\) are automatically path connected by Lemma~\ref{lem:periodic-connectedness}. Corollary~\ref{cor:quotient-complement} shows that \(X\setminus C\) is Oka as well.

For a related use of an exponential covering to treat a periodic obstacle, see \cite[Corollary~5.7]{Kus24}, concerning complements of small balls translated along a one-dimensional lattice.

Our applications concern convex tubes and products of planar compact sets. For a compact convex set \(B\subset\R^n\), define
\begin{equation}\label{eq:convex-tube}
T_B\coloneqq\R^n+\ii B=\{z\in\C^n:\im z\in B\}.
\end{equation}

\begin{theorem}\label{thm:main}\leavevmode
\begin{enumerate}[label=\textup{(\roman*)},leftmargin=*,itemsep=2pt]
\item For every \(n\geq2\) and every compact convex set \(B\subset\R^n\), the complement \(\C^n\setminus T_B\) is Oka.
\item For every \(n\geq2\) and compact sets \(K_1,\ldots,K_n\subset\C\) such that \(\C\setminus K_j\) has at most one bounded connected component for each \(j\), the complement \(\C^n\setminus(K_1\times\cdots\times K_n)\) is Oka. No regularity or connectedness assumption is imposed on the sets \(K_j\).
\end{enumerate}
In particular, \(\C^n\setminus\R^n\) is Oka for every \(n\geq2\).
\end{theorem}

Part~\textup{(ii)} includes products of closed disks, circles, and annuli. For \(\boldsymbol r=(r_1,\ldots,r_n)\) and \(\boldsymbol R=(R_1,\ldots,R_n)\), with \(0<r_j\leq R_j<\infty\), write
\begin{equation}\label{eq:annuli}
A_{\boldsymbol r,\boldsymbol R}
\coloneqq\{x\in\C^n:r_j\leq|x_j|\leq R_j,\ j=1,\ldots,n\}.
\end{equation}
Its complement is Oka, including the case of the standard product torus
\begin{equation}\label{eq:product-torus}
T\coloneqq\{x\in\C^n:|x_j|=1,\ j=1,\ldots,n\}
=A_{\boldsymbol1,\boldsymbol1}.
\end{equation}

When all planar factors in Theorem~\ref{thm:main}\,\textup{(ii)} are polynomially convex, their product is polynomially convex and its Oka complement already follows from \cite[Corollary~1.3]{Kus24}. The product formulation also permits a bounded complementary component in each factor.

For the constructions below, set
\begin{equation}\label{eq:tube}
T_\delta\coloneqq\{z\in\C^n:\norm{\im z}_\infty\leq\delta\},
\qquad \norm{y}_\infty\coloneqq\max_{1\leq j\leq n}|y_j|,
\quad \delta\geq0.
\end{equation}
Thus \(T_\delta=T_{[-\delta,\delta]^n}\) and \(T_0=\R^n\).

For the proof, we combine periodic divergence-free transport with attracting basins to construct entire holomorphic sprays. Kusakabe's proof of the complement theorem uses holomorphic families of nonautonomous attracting basins and their relative ellipticity to construct dominating sprays \cite[Lemma~4.4 and proof of Theorem~4.2]{Kus24}. The overall strategy builds on the work of Forstneri\v c and Wold, who constructed holomorphic families of Fatou--Bieberbach domains with prescribed centers and used them to obtain Oka complements \cite[Theorem~1.1]{FW20}, \cite[Theorem~2.3]{FW24}. The issue addressed here is uniform control of a noncompact forbidden set through its compact periodic quotient.

Lemma~\ref{lem:transport} transports a holomorphic section along a prescribed holomorphic path outside \(E\), while keeping the automorphisms uniformly close to the identity on all of \(E\). It also controls the first derivative near the section on a neighborhood independent of the approximation error. In dimension two, the local field is generated by one Hamiltonian. In arbitrary dimension, skew-symmetric coordinate-pair potentials give the required local fields. Each nonconstant Laurent mode has a complete flow; these are lifts of the multiplicative shears displayed by Anders\'en \cite[equation~(1)]{And00}, with the unused coordinates held fixed.

The finite-flow approximation is a volume-preserving Anders\'en--Lempert argument, originating in Anders\'en's work \cite{And90} and developed further by Anders\'en and Lempert \cite{AL92}, Forstneri\v c and Rosay \cite{FR93}, and Kutzschebauch in the parameter setting \cite{Kut05}; a survey is given in \cite{FK22}. Varolin established the volume density property of the multiplicative torus \cite[Corollary~4.5]{Var01}; see also Kaliman and Kutzschebauch \cite[Proposition~4.5]{KK10} for its algebraic form. Here, Laurent approximation of explicit scalar potentials and finite compositions of the corresponding complete flows provide the required volume-preserving approximation.

Whether \(X=(\C^*)^n\), \(n\geq2\), has the density property remains open \cite[Problem~2.9]{FK22}. Thus its known volume density property does not suffice for a direct application of Kusakabe's complement theorem \cite[Theorem~1.2]{Kus24}. Our construction combines volume-preserving transport on the quotient with an affine contraction on the covering space; the contraction need not descend to \(X\).

Compactness of \(C\) bounds the imaginary parts of every point of \(E\). We first construct one attracting basin outside a larger rectangular tube. Connectedness of \(Y\) and convexity of the parameter set provide a path from the prescribed section to a fixed center in that basin. Periodic transport, followed by an exact correction at the center, moves the obstacle into the larger tube. Pulling back the fixed basin gives the family in Theorem~\ref{thm:periodic-complement}. For positive-width rectangular tubes, Proposition~\ref{prop:family} adds strict forward invariance of the original tube. Its proof uses the parameter-uniform basin coordinates of Proposition~\ref{prop:basins}.

The normalized families satisfy Kusakabe's convex spray criterion \cite[Theorem~2.2]{Kus21}, which gives the Oka property within the spray framework initiated by Gromov \cite{Gro89}. For compact products, logarithmic covers of suitable Zariski open subsets of the complement provide the required families. Their sprays descend to these subsets, and Kusakabe's localization theorem \cite[Theorem~1.4]{Kus21} then applies to the whole complement. The planar hypothesis in Theorem~\ref{thm:main}\,\textup{(ii)} permits one puncture to meet every bounded complementary component; Lemma~\ref{lem:punctured-convexity} expresses this condition in terms of Runge convexity.

An affine complex change of coordinates extends Theorem~\ref{thm:main}\,\textup{(i)} to complements of maximally totally real affine planes and their compact convex transverse tubes (Corollary~\ref{cor:affine}). Part~\textup{(ii)} also includes products of Euclidean circles with arbitrary centers and positive radii.

After the analytic preliminaries in Section~\ref{sec:preliminaries}, Sections~\ref{sec:periodic} and~\ref{sec:basins} develop periodic transport and attracting basins to prove Theorem~\ref{thm:periodic-complement}. The applications, including Theorem~\ref{thm:main}, follow in Section~\ref{sec:applications}.

\section{Holomorphic approximation and sprays}\label{sec:preliminaries}

\subsection{Holomorphic convexity and scalar approximation}\label{subsec:stein-background}

We write \(\Ball_R\) for the open Euclidean ball of radius \(R\) centered at the origin in the ambient complex vector space, \(\OO(Z)\) for the algebra of holomorphic functions on a complex manifold \(Z\), and \(E\Subset U\) when \(\overline E\) is compact and contained in \(U\). Holomorphic dependence of a family \(A_\zeta\) means joint holomorphicity of \((\zeta,z)\mapsto A_\zeta(z)\). A vector field on a product \(B\times F\) is \emph{vertical} if it is tangent to the fibers of \(\operatorname{pr}_B\); vertical derivatives act only in the fiber variable. Below, \(F\) is either \(\C^d\) or \((\C^*)^n\).

For a nonempty compact set \(P\subset Z\), its holomorphic hull is
\[
\widehat P_{\OO(Z)}\coloneqq \{z\in Z:|h(z)|\leq\sup_P|h|\text{ for every }h\in\OO(Z)\}.
\]
The set \(P\) is \(\OO(Z)\)-convex if \(\widehat P_{\OO(Z)}=P\); the empty set is regarded as holomorphically convex. Compact convex subsets of \(\C^m\) are \(\OO(\C^m)\)-convex. Convex domains in \(\C^m\), \(\C^*\), and finite products of Stein manifolds are Stein \cite[Sections~2.1--2.2]{For17}.

The Oka--Weil theorem \cite[Theorem~2.3.1]{For17} states that if \(Z\) is Stein, \(P\subset Z\) is compact and \(\OO(Z)\)-convex, and \(h\) is holomorphic near \(P\), then \(h\) can be approximated uniformly on \(P\) by functions in \(\OO(Z)\). Such a compact set \(P\) has arbitrarily small compact \(\OO(Z)\)-convex neighborhoods \cite[Proposition~2.5.1, Corollary~2.5.3, and Proposition~2.5.5]{For17}. In particular, every open neighborhood \(U\) of \(P\) contains compact \(\OO(Z)\)-convex neighborhoods \(C_0,C_1\) with
\[
P\subset\Int C_0\subset C_0\subset\Int C_1\subset C_1\subset U.
\]

\subsection{Oka manifolds and dominating sprays}\label{subsec:oka-background}

The following criterion of Kusakabe characterizes the Oka property \cite[Theorem~2.2]{Kus21}. The spray condition below, denoted by \(\mathrm{C}\text{-}\mathrm{Ell}_1\), is the convex version of Gromov's condition \(\mathrm{Ell}_1\) \cite[Definition~3.1(b)--(c)]{For23}.

\begin{theorem}\label{thm:kusakabe}
A complex manifold \(Y\) is Oka if and only if the following condition holds. For every \(m\geq1\), every compact convex set \(K\subset\C^m\), and every holomorphic map \(f\) from a neighborhood of \(K\) to \(Y\), there exist a neighborhood \(V\) of \(K\) contained in the domain of \(f\), an integer \(N\geq1\), and a holomorphic map
\[
S:V\times\C^N\longrightarrow Y
\]
with \(S(\zeta,0)=f(\zeta)\), such that \(D_\xi S(\zeta,0)\) is surjective for every \(\zeta\in V\).
\end{theorem}

The map \(S\) in this criterion is called a \emph{dominating holomorphic spray over \(f\) with entire fiber} \(\C^N\).

Kusakabe's localization theorem takes the following form \cite[Theorem~1.4]{Kus21}.

\begin{theorem}\label{thm:localization}
If every point of a complex manifold has a Zariski open Oka neighborhood, then the manifold is Oka.
\end{theorem}

Here Zariski open means that the complement is a closed complex subvariety of the manifold, not necessarily an algebraic subvariety of the ambient Euclidean space.

\subsection{The covering and coordinate-pair fields}\label{subsec:flow-background}

Recall the covering \(q:\C^n\to X=(\C^*)^n\) in \eqref{eq:covering} and its deck group \(\Gamma=2\pi\Z^n\). The coordinates on the additive covering space and on \(X\) are denoted by \(z=(z_1,\ldots,z_n)\) and \(w=(w_1,\ldots,w_n)\), respectively. In Section~\ref{sec:applications}, the coordinates \(x=(x_1,\ldots,x_n)\) refer to the ambient space containing the compact product.

A scalar function or a vector-valued coefficient function \(h\) on \(\C^n\) is \(\Gamma\)-\emph{periodic} if \(h(z+\gamma)=h(z)\). A map \(A:\C^n\to\C^n\) is \(\Gamma\)-\emph{equivariant} if \(A(z+\gamma)=A(z)+\gamma\). Equivalently, its displacement \(A-\id\) is periodic. For a \(\Gamma\)-periodic holomorphic vector field \(W\), its local flow satisfies \(\Phi_t(z+\gamma)=\Phi_t(z)+\gamma\) wherever both sides are defined: the two curves solve \(\dot z=W(z)\) with the same initial value \(z+\gamma\), so they agree by uniqueness for the initial-value problem.

For a family of holomorphic scalar functions \(H=(H_{ij})_{1\leq i<j\leq n}\), set
\begin{equation}\label{eq:pair-field}
V_H\coloneqq\sum_{i<j}\left(
\frac{\partial H_{ij}}{\partial z_j}\frac{\partial}{\partial z_i}
-\frac{\partial H_{ij}}{\partial z_i}\frac{\partial}{\partial z_j}
\right).
\end{equation}
Mixed second derivatives cancel pairwise, so \(\operatorname{div}V_H=0\). Its local flow therefore has Jacobian determinant one. When \(n=2\), this is the Hamiltonian field with convention \(\iota_{X_h}(\dd z_1\wedge\dd z_2)=\dd h\), where \(h=H_{12}\).

The potentials required for a translation have an explicit form. If \(v\in\C^n\) is independent of the fiber coordinate and \(\eta=z-p\), take
\begin{equation}\label{eq:translation-potentials}
H_{ij}=\frac{v_i\eta_j-v_j\eta_i}{n-1}.
\end{equation}
The \(i\)-th component of \(V_H\) receives \(v_i/(n-1)\) from each of the \(n-1\) pairs containing \(i\), hence \(V_H=v\). This identity also holds when \(p\) and \(v\) depend holomorphically on parameters.

\section{Periodic divergence-free transport}\label{sec:periodic}

The covering \(q\), quotient \(X\), and lattice \(\Gamma\) are those of Section~\ref{subsec:flow-background}.

\begin{lemma}\label{lem:transport}
Suppose \(C\subset X\) is compact and \(\OO(X)\)-convex, \(U\subset\C^m\) is convex and open, and \(L\Subset U\) is compact and convex. Set \(E=q^{-1}(C)\). For a holomorphic map \(p:\C\times U\to\C^n\) with \(p(t,\zeta)\notin E\) for \(t\in[0,1]\) and \(\zeta\in L\), there is \(r>0\), independent of \(\epsilon\), such that for every \(\epsilon>0\) there is a holomorphic family of \(\Gamma\)-equivariant automorphisms \(A_\zeta\), \(\zeta\in U\), of \(\C^n\) satisfying
\begin{align*}
\sup_{\zeta\in L,\;z\in E}|A_\zeta(z)-z|&<\epsilon,\\
\sup_{\zeta\in L,\;|u|\leq r}
|A_\zeta(p(0,\zeta)+u)-p(1,\zeta)-u|&<\epsilon,\\
\sup_{\zeta\in L,\;|u|\leq r}
\norm{D_zA_\zeta(p(0,\zeta)+u)-I_n}&<\epsilon.
\end{align*}
The inverse family is holomorphic, and \(\det D_zA_\zeta=1\). The statement also applies when \(p\) has no parameter \(\zeta\).
\end{lemma}

We prove the lemma after establishing the scalar and finite-flow approximation needed below.

\subsection{Approximation on the periodic quotient}\label{subsec:periodic-approximation}

The following is the disjoint-set case of Forstneri\v c's holomorphic convexity lemma \cite[Lemma~6.5, p.~111]{For99}. The Stein version follows as well by a proper Stein embedding and Cartan extension: holomorphic convexity on a closed Stein submanifold agrees with ambient polynomial convexity for compact sets \cite[Chapter~2]{For17}.

\begin{lemma}\label{lem:union}
Let \(Z\) be Stein, \(G\subset Z\) be a closed complex submanifold, and \(P\subset Z\setminus G\) be compact and \(\OO(Z)\)-convex. If \(Q\subset G\) is compact and \(\OO(G)\)-convex, then \(P\cup Q\) is \(\OO(Z)\)-convex.
\end{lemma}

Notice that Lemma~\ref{lem:union} requires \(P\cap G=\varnothing\), not only \(P\cap Q=\varnothing\).

Write \(w=(w_1,\ldots,w_n)\) for the coordinates on \(X\). For \(k=(k_1,\ldots,k_n)\in\Z^n\), write \(w^k\coloneqq \prod_{j=1}^n w_j^{k_j}\), and set
\begin{equation}\label{eq:D-operators}
\mathcal D_j\coloneqq \ii w_j\frac{\partial}{\partial w_j},\qquad j=1,\ldots,n.
\end{equation}
These differential operators act in the quotient variable \(w\) and commute. On a product \(B\times X\), they are vertical for \(\operatorname{pr}_B:B\times X\to B\). For a multi-index \(\beta=(\beta_1,\ldots,\beta_n)\in\Z_{\geq0}^n\), we set \(\mathcal D^\beta\coloneqq \mathcal D_1^{\beta_1}\cdots\mathcal D_n^{\beta_n}\). A scalar function pulled back from \(X\) is \(\Gamma\)-periodic, and
\begin{equation}\label{eq:pullback-derivative}
\frac{\partial}{\partial z_j}(H\circ q) =(\mathcal D_j H)\circ q.
\end{equation}

\begin{lemma}\label{lem:laurent}
Let \(B\) be a Stein manifold, \(S\subset B\times X\) be compact and \(\OO(B\times X)\)-convex, and \(H\) be holomorphic on an open neighborhood \(\mathcal N\) of \(S\). Then there is a compact neighborhood \(C_0\) of \(S\) contained in \(\mathcal N\) with the following property: for every \(\eta>0\), there is a finite Laurent sum
\begin{equation}\label{eq:Laurent-sum}
P(b,w)\coloneqq \sum_{k\in\Lambda}c_k(b)w^k, \qquad\Lambda\subset\Z^n\text{ finite},\quad c_k\in\OO(B),
\end{equation}
such that
\begin{equation}\label{eq:Laurent-jets}
\max_{|\beta|\leq2}\sup_{C_0} \left|\mathcal D^\beta(P-H)\right|<\eta.
\end{equation}
\end{lemma}

\begin{proof}
By \cite[Proposition~2.5.1, Corollary~2.5.3, and Proposition~2.5.5]{For17}, choose compact \(\OO(B\times X)\)-convex neighborhoods \(C_0,C_1\), independently of \(\eta\), so that
\begin{equation}\label{eq:nested-neighborhoods}
S\subset\Int C_0\subset C_0\subset\Int C_1 \subset C_1\subset\mathcal N.
\end{equation}
By the Oka--Weil theorem \cite[Theorem~2.3.1]{For17}, \(H\) can be approximated uniformly on \(C_1\) by a function \(h\in\OO(B\times X)\). Cauchy estimates in finitely many product coordinate charts show that the error in the derivatives \(\mathcal D^\beta\), \(|\beta|\leq2\), on \(C_0\) can be made smaller than \(\eta/2\). The constants in these estimates depend only on the fixed nested neighborhoods and charts.

The global function \(h\) has a Laurent expansion
\begin{equation}\label{eq:global-laurent}
h(b,w)=\sum_{k\in\Z^n}c_k(b)w^k,
\end{equation}
where
\begin{equation}\label{eq:laurent-coefficients}
c_k(b)=\frac{1}{(2\pi\ii)^n}\int_{|v_1|=1}\!\cdots\!\int_{|v_n|=1}\frac{h(b,v_1,\ldots,v_n)}{\prod_{j=1}^n v_j^{k_j+1}}\,\dd v_n\cdots\dd v_1.
\end{equation}
Local differentiation under the integral shows that the coefficients are holomorphic in \(b\). They are independent of the positive integration radii.

To verify uniform convergence, project a given compact subset of \(B\times X\) to a compact subset of \(B\) and enclose its \(w\)-projection in a product of closed annuli. Choose slightly smaller inner radii and slightly larger outer radii. In \eqref{eq:laurent-coefficients}, use an outer radius when the corresponding exponent is nonnegative and an inner radius when it is negative. The resulting Cauchy estimates bound the terms by
\[
C\,\prod_{j=1}^n\theta_j^{|k_j|},\qquad 0<\theta_j<1,
\]
uniformly on the prescribed compact set. Application of \(\mathcal D^\beta\) multiplies a term by \(\prod_{j=1}^n(\ii k_j)^{\beta_j}\), which preserves summability. Thus the series and the indicated derivatives converge normally. A finite truncation of \eqref{eq:global-laurent} approximates \(h\) with these derivatives on \(C_0\) to within \(\eta/2\). Combining the two approximations proves the claim. The neighborhoods in \eqref{eq:nested-neighborhoods} depend only on \(S\) and \(\mathcal N\), so the same neighborhoods can be used for any prescribed finite family of scalar functions on \(\mathcal N\).
\end{proof}

The scalar approximation on the Stein quotient will be converted into automorphisms using complete \(\Gamma\)-periodic coordinate-pair fields. The following flows are lifts of special cases of Anders\'en's multiplicative shears \cite[equation~(1)]{And00}. We record their formulas and holomorphic parameter dependence for use below.

\begin{lemma}\label{lem:complete}
Let \(B\) be a complex manifold, \(c\in\OO(B)\), \(k\in\Z^n\setminus\{0\}\), and \(1\leq i<j\leq n\). The potential \(H_{ij}(b,z)=c(b)e^{\ii k\cdot z}\), with the other pair potentials zero, gives the complete vertical field
\begin{equation}\label{eq:mode-field}
V_{ij,k}(b,z)=\ii c(b)e^{\ii k\cdot z}v_{ij,k},
\qquad v_{ij,k}\coloneqq k_je_i-k_ie_j.
\end{equation}
For every complex time \(\tau\), its flow is
\begin{equation}\label{eq:mode-flow}
\Phi_{ij,k,\tau}^{b}(z)
=z+\tau\ii c(b)e^{\ii k\cdot z}v_{ij,k}.
\end{equation}
This flow is holomorphic in \((b,\tau,z)\), is an automorphism in \(z\), has Jacobian determinant one, and satisfies
\begin{equation}\label{eq:mode-equivariance}
\Phi_{ij,k,\tau}^{b}(z+\gamma)
=\Phi_{ij,k,\tau}^{b}(z)+\gamma,\qquad\gamma\in\Gamma.
\end{equation}
\end{lemma}

\begin{proof}
The formula for the vector field follows by differentiation. Since \(k\cdot v_{ij,k}=0\), the function \(e^{\ii k\cdot z}\) is constant along the proposed trajectory. Hence \eqref{eq:mode-flow} solves the flow equation for every \(\tau\in\C\). The inverse is \(\Phi_{ij,k,-\tau}^{b}\). Its derivative is
\[
I_n-\tau c(b)e^{\ii k\cdot z}v_{ij,k}k^{\mathsf T}.
\]
Since \(\det(I+uv^{\mathsf T})=1+v^{\mathsf T}u\) and \(k^{\mathsf T}v_{ij,k}=0\), this derivative has determinant one. Finally, \(e^{\ii k\cdot\gamma}=1\) for \(\gamma\in2\pi\Z^n\), which gives \eqref{eq:mode-equivariance}.
\end{proof}

On the quotient the only changed coordinates are
\[
w_i\longmapsto w_i\exp(-\tau c(b)k_jw^k),\qquad
w_j\longmapsto w_j\exp(\tau c(b)k_iw^k).
\]
For \(n=2\), these are exactly the cited shears. In higher dimension, the remaining coordinates are fixed and enter as holomorphic parameters in their coefficients.

A constant term in any Laurent expansion contributes the zero vector field. Terms with \(k_i=k_j=0\) also contribute zero and may be omitted. Thus a field obtained from finitely many pulled-back Laurent polynomials is a finite sum of the complete fields in Lemma~\ref{lem:complete}. The next lemma is the quantitative finite-composition step used in Anders\'en--Lempert approximation \cite{AL92,FR93,Kut05}. The parameter-uniform first-derivative estimate is needed on a fixed neighborhood of the section.

For the next lemma and its proof, let \(I\subset\C\) be an open neighborhood of \([0,1]\), \(V\) be a complex manifold, and
\[
Y_j:I\times V\times\C^d\longrightarrow\C^d, \qquad j=1,\ldots,\ell,
\]
be jointly holomorphic vertical vector fields. Assume each frozen field \(Y_j(s,\zeta,\cdot)\) is complete, with flow \(\phi_{j,\tau}^{s,\zeta}\) jointly holomorphic in \((s,\zeta,\tau,z)\in I\times V\times\C\times\C^d\). The evolution considered below uses only real time \(t\in[0,1]\).

Set \(Y\coloneqq \sum_{j=1}^{\ell}Y_j\). For \(N\geq1\), set \(\Delta\coloneqq 1/N\), \(t_j\coloneqq j/N\), and
\begin{align}
B_{j,N,\zeta} &\coloneqq \phi_{\ell,\Delta}^{t_j,\zeta}\circ\cdots\circ \phi_{1,\Delta}^{t_j,\zeta},\label{eq:split-step}\\
A_{N,\zeta} &\coloneqq B_{N-1,N,\zeta}\circ\cdots\circ B_{0,N,\zeta}. \label{eq:split-global}
\end{align}

\begin{lemma}\label{lem:splitting}
Under the preceding assumptions, let \(L\subset V\) and \(E\subset L\times\C^d\) be compact. Suppose the nonautonomous equation
\begin{equation}\label{eq:nonautonomous-general}
\dot z(t)=Y(t,\zeta,z(t))
\end{equation}
has a solution for every initial pair \((\zeta,z(0))\in E\) on \([0,1]\), with the union of these trajectories contained in a compact set. Denote its time-one map, on a neighborhood of these initial points, by \(\Theta_\zeta\).

Then \(A_{N,\zeta}\) is a holomorphic family of automorphisms, and
\begin{equation}\label{eq:split-error}
\sup_{(\zeta,z)\in E} \left( |A_{N,\zeta}(z)-\Theta_\zeta(z)| +\norm{D_zA_{N,\zeta}(z)-D_z\Theta_\zeta(z)} \right)\longrightarrow0.
\end{equation}
In fact, the error is \(O(N^{-1})\).
\end{lemma}

\begin{proof}
Each composition in \eqref{eq:split-global} is an automorphism and depends holomorphically on the parameter. We compare it with the exact nonautonomous flow near the prescribed compact family of trajectories.

For the one-step estimate, choose a compact neighborhood of the prescribed trajectories with a positive spatial buffer. On this neighborhood, the finitely many fields and their derivatives needed below are uniformly bounded in \(t\) and \(\zeta\). Local existence and continuous dependence for ordinary differential equations give a common neighborhood of the initial compact set on which the exact solutions exist up to time one. All estimates below are made on a slightly smaller buffered neighborhood.

For sufficiently small \(\Delta\), Taylor's formula for a frozen flow gives, uniformly there,
\[
\phi_{j,\Delta}^{t,\zeta}(z) =z+\Delta Y_j(t,\zeta,z)+O(\Delta^2).
\]
Composing a fixed finite number of these formulas gives
\begin{equation}\label{eq:one-step-expansion}
B_{j,N,\zeta}(z) =z+\Delta Y(t_j,\zeta,z)+O(\Delta^2).
\end{equation}
The exact nonautonomous transition from \(t_j\) to \(t_j+\Delta\) has the same expansion, because the field is continuously differentiable in time. Its one-step difference from \eqref{eq:one-step-expansion} is therefore at most \(C_1\Delta^2\). The derivative of the split step is bounded by \(1+C_2\Delta\).

Let \(e_j\) be the maximal error after \(j\) steps. As long as the approximate trajectory stays in the buffered neighborhood, it satisfies
\begin{equation}\label{eq:error-recursion}
e_{j+1}\leq(1+C_2\Delta)e_j+C_1\Delta^2, \qquad e_0=0.
\end{equation}
Iterating the recurrence and using \(j\Delta\leq1\) gives
\[
e_j\leq C_1\Delta^2\sum_{r=0}^{j-1}(1+C_2\Delta)^r \leq C\Delta\qquad(0\leq j\leq N),
\]
where \(C\) is independent of \(N\) and the initial pair. Taking \(N\) sufficiently large makes this bound smaller than the buffer. A first-exit argument therefore gives the estimate at every step. Each individual intermediate frozen flow also remains there for large \(N\), since its displacement on that region is \(O(\Delta)\).

For the derivative estimate, adjoin the matrix variable \(J\in\operatorname{Mat}_{d\times d}(\C)\) and use the lifted fields
\[
\widetilde Y_j(t,\zeta,z,J) \coloneqq \bigl(Y_j(t,\zeta,z),\ D_zY_j(t,\zeta,z)J\bigr).
\]
Their complete flows are
\[
(z,J)\longmapsto \bigl(\phi_{j,\tau}^{t,\zeta}(z),\ D_z\phi_{j,\tau}^{t,\zeta}(z)J\bigr).
\]
The exact lifted equation is the original equation together with its variational equation. Starting with \(J=I_d\), its trajectories are uniformly bounded, by Gronwall's inequality and the bound on \(D_zY\). The preceding error argument on the lifted space therefore proves the \(O(\Delta)\) estimate for both the map and its derivative. This proves \eqref{eq:split-error}.
\end{proof}

We apply Lemma~\ref{lem:splitting} at one sufficiently large finite \(N\). Thus the approximating map is an automorphism, regardless of whether the sum of the fields is complete.

\subsection{Periodic transport}\label{subsec:transport-tube}

\begin{proof}[Proof of Lemma~\ref{lem:transport}]
Set \(B=\C\times U\), \(J=[0,1]\times L\), and \(Z=B\times X\). In the Stein manifold \(Z\), consider
\[
P=J\times C,\qquad
G=\{(s,\zeta,q(p(s,\zeta))):(s,\zeta)\in B\},\qquad
Q_0=G|_J.
\]
The compact convex set \(J\) is \(\OO(B)\)-convex, so \(P\) is \(\OO(Z)\)-convex. The graph \(G\) is a closed complex submanifold biholomorphic to \(B\), and \(Q_0\) is \(\OO(G)\)-convex. A point of \(P\cap G\) would have \((s,\zeta)\in J\) and \(p(s,\zeta)\in E\), contrary to the avoidance hypothesis. Lemma~\ref{lem:union} therefore shows that \(S=P\cup Q_0\) is \(\OO(Z)\)-convex.

\smallskip
\noindent\textbf{Step 1: Construction and approximation of the local potentials.}
Near \(Q_0\), use the branch of the logarithm near \(1\) with \(\operatorname{Log}1=0\), and set
\[
\eta_j(s,\zeta,w)=\frac1{\ii}\operatorname{Log}\bigl(w_je^{-\ii p_j(s,\zeta)}\bigr),
\qquad j=1,\ldots,n.
\]
On disjoint neighborhoods of \(P\) and \(Q_0\), define, for \(i<j\),
\begin{equation}\label{eq:local-pair-potentials}
\mathcal H_{ij}(s,\zeta,w)=
\begin{cases}
0,&\text{near }P,\\[2pt]
\displaystyle\frac{\partial_s p_i(s,\zeta)\eta_j(s,\zeta,w)-\partial_s p_j(s,\zeta)\eta_i(s,\zeta,w)}{n-1},&\text{near }Q_0.
\end{cases}
\end{equation}
After pullback by \(q\), the field \eqref{eq:pair-field} vanishes near \(E\) and equals \(\partial_s p(s,\zeta)\) near \(p(s,\zeta)\), by \eqref{eq:translation-potentials}. The neighborhoods are chosen so that all these scalar functions are holomorphic on the same open neighborhood of \(S\).

Apply Lemma~\ref{lem:laurent} to this finite family, obtaining a common fixed compact neighborhood \(K_0\) inside its domain. Compactness of \(P\) and \(Q_0\) gives \(\rho,\sigma>0\) such that, for every \((t,\zeta)\in J\), the quotient images of
\[
\{p(t,\zeta)+u:|u|\leq\rho\}
\quad\text{and}\quad
\{z:\dist(z,E)\leq\sigma\}
\]
lie in the corresponding fiber of \(K_0\), in the translation and zero regions, respectively. For the second set this follows from compactness modulo \(\Gamma\): restrict the real parts to a closed fundamental box and use periodicity. Decrease \(\rho\) so that \(\operatorname{Log}(e^{\ii u_j})=\ii u_j\) on these balls, and set \(r=\rho/2\). These choices are independent of the approximation error.

For any \(\alpha>0\), approximate each \(\mathcal H_{ij}\) by
\[
\mathcal P_{ij}(s,\zeta,w)=\sum_{k\in\Lambda_{ij}}c_{ij,k}(s,\zeta)w^k,
\qquad \Lambda_{ij}\subset\Z^n\text{ finite},\quad c_{ij,k}\in\OO(B).
\]
To make the dimension dependence explicit, if the scalar errors in all derivatives \(\mathcal D^\beta\), \(|\beta|\leq2\), are at most \(\eta\), then each component of the field error is at most \((n-1)\eta\), and each entry of its derivative matrix is at most \((n-1)\eta\). The Euclidean norm and operator norm are consequently bounded by \(\sqrt n(n-1)\eta\) and \(n(n-1)\eta\), respectively. Choose
\[
\eta<\frac{\alpha}{(\sqrt n+n)(n-1)}.
\]
Then the field
\[
W(t,\zeta,z)=V_{(\mathcal P_{ij}(t,\zeta,q(z)))_{i<j}}(z)
\]
satisfies
\begin{equation}\label{eq:transport-zero}
|W(t,\zeta,z)|+\norm{D_zW(t,\zeta,z)}<\alpha
\qquad\bigl(\dist(z,E)\leq\sigma\bigr),
\end{equation}
\begin{equation}\label{eq:transport-moving}
\begin{split}
&|W(t,\zeta,p(t,\zeta)+u)-\partial_t p(t,\zeta)|\\
&\qquad+\norm{D_zW(t,\zeta,p(t,\zeta)+u)}<\alpha
\qquad(|u|\leq\rho),
\end{split}
\end{equation}
uniformly for \((t,\zeta)\in J\). The first derivatives of the potentials control the field and the second derivatives control \(D_zW\), by \eqref{eq:pullback-derivative}. Periodicity gives \eqref{eq:transport-zero} on the entire indicated neighborhood of \(E\).

\smallskip
\noindent\textbf{Step 2: Control of the local flow.}
Take \(\alpha<\min\{\sigma/2,\rho/4,1\}\). For a solution of \(\dot z=W(t,\zeta,z)\) starting in \(E\), integration up to a hypothetical first exit from its \(\sigma\)-neighborhood gives \(|z(t)-z(0)|\leq\alpha t\), so exit before time one is impossible. For a solution starting at \(p(0,\zeta)+u\), \(|u|\leq r\), subtract the reference path \(p(t,\zeta)+u\). Equation~\eqref{eq:transport-moving} gives
\[
|z(t)-p(t,\zeta)-u|\leq\alpha t.
\]
Since \(r+\alpha<3\rho/4\), this solution cannot leave the moving ball of radius \(\rho\). Each solution stays in a bounded neighborhood of its reference path and therefore exists up to time one.

Write \(\Theta_{t,\zeta}\) for these local flow maps. The strict margins also give local existence for initial points in a neighborhood of the closed sets under consideration. Their variational equation and~\eqref{eq:transport-moving} imply
\[
\norm{D_z\Theta_{1,\zeta}(p(0,\zeta)+u)-I_n}
\leq e^\alpha-1,
\qquad |u|\leq r,\quad\zeta\in L.
\]

\smallskip
\noindent\textbf{Step 3: Approximation of the local flow by global automorphisms.}
Every nonconstant Laurent mode has the complete lifted flow in Lemma~\ref{lem:complete}. Apply Lemma~\ref{lem:splitting} to these finitely many modes, taking the compact set of initial pairs
\[
(L\times D_E)\ \cup\
\{(\zeta,p(0,\zeta)+u):\zeta\in L,\ |u|\leq r\},
\]
where
\[
D_E=\{z\in E:0\leq\re z_j\leq2\pi,\ j=1,\ldots,n\}.
\]
The set \(D_E\) is compact because \(C\subset X\) is compact. The preceding bounds and continuous dependence of the flow show that the exact trajectories form a compact family. Lemma~\ref{lem:splitting} therefore gives a finite composition \(A_\zeta\) approximating \(\Theta_{1,\zeta}\), together with its \(z\)-derivative, uniformly on these initial pairs. Uniqueness makes the exact flow equivariant wherever defined, and every approximating factor is equivariant. Thus the approximation on \(D_E\) extends to all of \(E\).

Given \(\epsilon>0\), first choose \(\alpha\) so that \(\alpha<\epsilon/3\) and \(e^\alpha-1<\epsilon/3\), and fix the corresponding Laurent polynomial. Then choose the finite subdivision in Lemma~\ref{lem:splitting} so that the map and derivative errors are less than \(\epsilon/3\). Combining these estimates gives all three conclusions. Equivariance, determinant one, and holomorphic dependence of the map and its inverse follow from the corresponding properties of the complete mode flows.
\end{proof}

For the standard volume form \(\omega=\dd z_1\wedge\cdots\wedge\dd z_n\), the contraction \(\iota_{V_H}\omega\) is exact in the fiber variable: it is the fiber differential of
\[
\sum_{i<j}(-1)^{i+j-3}H_{ij}\,
\dd z_1\wedge\cdots\wedge\widehat{\dd z_i}\wedge\cdots\wedge\widehat{\dd z_j}\wedge\cdots\wedge\dd z_n.
\]

\section{Attracting basins and periodic complements}\label{sec:basins}

We first establish connectedness of the periodic complements and prepare a fixed attracting basin and holomorphic paths. These ingredients, together with periodic transport, prove Theorem~\ref{thm:periodic-complement}; the quotient complement is then shown to be Oka. Finally, we construct families for positive-width rectangular tubes that make the original tube strictly forward invariant, using parameter-uniform basin coordinates.

A fixed point \(p\) of \(g\in\Aut(\C^d)\) is \emph{attracting} if every eigenvalue of \(Dg(p)\) has modulus less than one. Its \emph{attracting basin} is \(\{z\in\C^d:g^k(z)\to p\text{ as }k\to\infty\}\). A set \(E\subset\C^d\) is \emph{forward invariant} under \(g\) if \(g(E)\subset E\).

\subsection{Geometric preparations}\label{subsec:basin-models}

\begin{lemma}\label{lem:periodic-connectedness}
For \(n\geq2\), let \(C\subset X=(\C^*)^n\) be compact and \(\OO(X)\)-convex. Then both \(X\setminus C\) and \(\C^n\setminus q^{-1}(C)\) are path connected.
\end{lemma}

\begin{proof}
The assertion is immediate if \(C=\varnothing\), so assume \(C\neq\varnothing\). Choose \(R>0\) so large that \(C\) lies in the interior of the compact polyannulus
\[
P_R=q(T_R)=\{w\in X:e^{-R}\leq|w_j|\leq e^R,\ j=1,\ldots,n\}.
\]
Logarithmic polar coordinates identify \(X\setminus P_R\) with
\[
(\mathbb S^1)^n\times\bigl(\R^n\setminus[-R,R]^n\bigr),
\]
which is path connected because \(n\geq2\). Hence \(X\setminus P_R\) lies in one connected component of \(X\setminus C\). Any other component \(D\) would be contained in \(P_R\), so it would be relatively compact in \(X\), with \(\partial D\subset C\). The maximum modulus principle gives
\[
|h(w)|\leq\sup_C|h|\qquad(w\in D,\ h\in\OO(X)).
\]
Thus \(D\) lies in the \(\OO(X)\)-hull of \(C\), which equals \(C\), a contradiction. Therefore \(X\setminus C\) is connected and, being open in a manifold, is path connected.

Set \(Y=\C^n\setminus q^{-1}(C)\). For any \(z\in Y\), choose a path in \(X\setminus C\) from \(q(z)\) to a point of \(X\setminus P_R\). Its lift under the covering \(q:Y\to X\setminus C\), starting at \(z\), ends in
\[
q^{-1}(X\setminus P_R)=\C^n\setminus T_R
\cong\R^n\times\bigl(\R^n\setminus[-R,R]^n\bigr).
\]
This set is path connected because \(n\geq2\). Every point of \(Y\) can therefore be joined to the same path-connected subset of \(Y\), proving that \(Y\) is path connected.
\end{proof}

A Fatou--Bieberbach domain is a proper subdomain of \(\C^n\) biholomorphic to \(\C^n\); constructing one outside a tube therefore gives an entire parametrization whose image avoids that tube. In the proof of Theorem~\ref{thm:periodic-complement}, a holomorphic family of automorphisms will send the prescribed section to the center of this fixed model while keeping the image of the obstacle inside the tube. Pulling the model back and normalizing its parametrizations then gives the dominating sprays needed for the Oka property. This construction follows the strategy of Forstneri\v c and Wold \cite[Lemmas~3.2--3.3 and proof of Theorem~3.1]{FW20}. We construct the model as an attracting basin so that the basin theorem supplies its biholomorphism with \(\C^n\), while forward invariance of the closed tube guarantees that the basin avoids the tube when the attracting point lies outside it. Conjugating the dynamics also preserves the basin description in the resulting family.

\begin{lemma}\label{lem:tube}
Let \(n\geq2\). For every \(R>0\) and every real \(h>4R\), there is a Fatou--Bieberbach domain \(\Omega\subset\C^n\setminus T_R\) and a biholomorphism \(\Phi:\C^n\to\Omega\) with \(\Phi(0)=(\ii h,\ldots,\ii h)\) and \(D\Phi(0)=I_n\). The domain \(\Omega\) is the attracting basin of an automorphism \(\Psi\) fixing this point.
\end{lemma}

\begin{proof}
Write \(p_*=(\ii h,\ldots,\ii h)\). The set
\[
q( T_R)
=\{w\in X:e^{-R}\leq|w_j|\leq e^R,\ j=1,\ldots,n\}
\]
is \(\OO(X)\)-convex, as the functions \(w_j\) and \(w_j^{-1}\) separate every point outside it. Choose
\(0<\epsilon<\min\{R/4,1/10\}\).
The path \(p(s)=(1+s)p_*\) avoids \( T_R\) for \(s\in[0,1]\), so Lemma~\ref{lem:transport}, without parameters, gives \(A\in\Aut(\C^n)\) satisfying
\[
\sup_{ T_R}|A-\id|<\epsilon,\qquad
|A(p_*)-2p_*|<\epsilon,\qquad
\norm{DA(p_*)-I}<\epsilon.
\]
The automorphism
\[
\Psi(z)=p_*+\tfrac12\bigl(A(z)-A(p_*)\bigr)
\]
fixes \(p_*\). For \(z\in T_R\), the identity
\[
\Psi(z)=\tfrac12 z+\tfrac12\bigl(A(z)-z\bigr)
+\tfrac12\bigl(2p_*-A(p_*)\bigr)
\]
and the preceding estimates give
\[
\Psi( T_R)\subset T_{R/2+\epsilon}\subset T_R,
\qquad
\norm{D\Psi(p_*)}<\tfrac12(1+\epsilon)<\tfrac{11}{20}<1.
\]
All eigenvalues of \(D\Psi(p_*)\) consequently have modulus less than one. Its attracting basin
\[
\Omega=\{z\in\C^n:\Psi^\nu(z)\longrightarrow p_*\text{ as }\nu\longrightarrow\infty\}
\]
is biholomorphic to \(\C^n\) by Rosay and Rudin \cite[Appendix, pp.~84--85]{RR88}, after translating \(p_*\) to the origin. No orbit in the closed forward-invariant set \( T_R\) can converge to \(p_*\notin T_R\). Thus \(\Omega\) avoids \(T_R\). A translation in the source first normalizes \(\Phi(0)=p_*\); precomposition with the inverse of its derivative at \(0\) then gives \(D\Phi(0)=I_n\), without changing the image.
\end{proof}

To transport a parameter family to the fixed basin constructed above, we need a holomorphic path to its center. The following construction applies to any connected open complement. Only the real-time path is required to avoid the obstacle; its polynomial extension in complex time need not do so.

\begin{lemma}\label{lem:holomorphic-path}
Suppose \(E\subset\C^n\) is a nonempty closed set and \(Y=\C^n\setminus E\) is connected. Let \(U\subset\C^m\) be convex and open, \(L\Subset U\) be nonempty, compact, and convex, \(f:U\to Y\) be holomorphic, and \(p_*\in Y\). There is a holomorphic map \(p:\C\times U\to\C^n\), polynomial in its first variable, such that
\[
p(0,\zeta)=f(\zeta),\qquad p(1,\zeta)=p_*
\quad(\zeta\in U),
\]
and \(p(t,\zeta)\notin E\) for \(t\in[0,1]\) and \(\zeta\in L\).
\end{lemma}

\begin{proof}
Choose \(\zeta_0\in L\). Since an open connected subset of Euclidean space is path connected, choose a continuous path \(\gamma:[0,1]\to Y\) with \(\gamma(0)=f(\zeta_0)\) and \(\gamma(1)=p_*\). Define
\[
P(t,\zeta)=
\begin{cases}
f((1-2t)\zeta+2t\zeta_0),&0\leq t\leq\tfrac12,\\
\gamma(2t-1),&\tfrac12\leq t\leq1.
\end{cases}
\]
Convexity of \(U\) makes the first expression well defined for every \(\zeta\in U\). The two expressions agree at \(t=1/2\), so \(P\) is jointly continuous and is holomorphic in \(\zeta\) at each time. Its values lie in \(Y\). In particular,
\[
d\coloneqq\dist\bigl(P([0,1]\times L),E\bigr)>0,
\]
since the first set is compact and \(E\) is closed.

The Bernstein polynomials
\begin{equation}\label{eq:bernstein-path}
p_N(s,\zeta)=\sum_{j=0}^N\binom Nj s^j(1-s)^{N-j}P(j/N,\zeta)
\end{equation}
are holomorphic on \(\C\times U\) and retain the endpoints exactly. Their convergence to \(P\) on \([0,1]\times L\) is uniform in the parameter. To verify the uniformity explicitly, let
\[
M=\sup_{[0,1]\times L}|P|,
\qquad
\omega(\delta)=\sup_{\substack{|t-t'|\leq\delta\\t,t'\in[0,1],\ \zeta\in L}}
|P(t,\zeta)-P(t',\zeta)|.
\]
The Bernstein weights are nonnegative and sum to one. Their second moment about \(t\) is \(t(1-t)/N\), so the total weight of indices with \(|j/N-t|>\delta\) is at most \(1/(4N\delta^2)\). Splitting the sum accordingly gives
\[
\sup_{[0,1]\times L}|p_N-P|
\leq\omega(\delta)+\frac{M}{2N\delta^2}.
\]
First choose \(\delta>0\) so that \(\omega(\delta)<d/4\), and then choose \(N\) so large that the second term is less than \(d/4\). It follows that \(p=p_N\) stays at distance at least \(d/2\) from \(E\) on \([0,1]\times L\), as required.
\end{proof}

We can now prove the general periodic-complement theorem.

\subsection{Periodic complements and Fatou--Bieberbach families}\label{subsec:periodic-families}

\begin{proof}[Proof of Theorem~\ref{thm:periodic-complement}]
It suffices to construct the stated normalized family, since its fiber derivative is surjective and Theorem~\ref{thm:kusakabe} then gives the Oka property. We may assume that the parameter set \(K\) is nonempty.

\smallskip
\noindent\textbf{Step 1: Choice of a fixed model outside a larger tube.}
Compactness of \(C\) gives numbers \(0<r_j\leq R_j<\infty\) such that \(r_j\leq|w_j|\leq R_j\) on \(C\). If \(z\in E\), then \(|e^{\ii z_j}|=e^{-\im z_j}\), hence
\[
-\log R_j\leq\im z_j\leq-\log r_j.
\]
Consequently \(E\subset T_{R_0}\) for some \(R_0>0\). Set \(S=R_0+1\) and choose \(h>4S\). Lemma~\ref{lem:tube} gives an automorphism \(\Psi\) with attracting fixed point \(p_*=(\ii h,\ldots,\ii h)\), its basin \(\Omega\), and a biholomorphism
\begin{equation}\label{eq:fixed-model}
\Phi:\C^n\overset{\sim}{\longrightarrow}\Omega\subset\C^n\setminus T_S,
\qquad \Phi(0)=p_*,\quad D\Phi(0)=I_n.
\end{equation}
This choice is independent of the section \(f\).

\smallskip
\noindent\textbf{Step 2: Transport to the fixed model and exact correction of the center.}
Choose a compact convex set \(L\) and a bounded convex open set \(U\) such that
\[
K\subset\Int L\subset L\Subset U\Subset U_0,
\]
and set \(V=\Int L\). Sufficiently small Euclidean thickenings of \(K\) give these sets. Lemma~\ref{lem:periodic-connectedness} shows that \(Y\) is path connected, so Lemma~\ref{lem:holomorphic-path} gives a holomorphic path \(p\) from \(f\) to \(p_*\) avoiding \(E\) over \([0,1]\times L\). Apply Lemma~\ref{lem:transport} with \(0<\epsilon<1/4\). It provides a holomorphic family of automorphisms \(A_\zeta\), \(\zeta\in U\), with holomorphic inverses, satisfying
\[
\sup_{\zeta\in L,\ z\in E}|A_\zeta(z)-z|<\epsilon,
\qquad
\sup_{\zeta\in L}|A_\zeta(f(\zeta))-p_*|<\epsilon.
\]
The correction
\begin{equation}\label{eq:general-correction}
\widetilde A_\zeta(z)=A_\zeta(z)+p_*-A_\zeta(f(\zeta))
\end{equation}
satisfies \(\widetilde A_\zeta(f(\zeta))=p_*\), and
\begin{equation}\label{eq:obstacle-in-model-tube}
\sup_{\zeta\in L,\ z\in E}|\widetilde A_\zeta(z)-z|<2\epsilon<\tfrac12,
\qquad \widetilde A_\zeta(E)\subset T_S.
\end{equation}
The inverse family is explicitly
\[
\widetilde A_\zeta^{-1}(y)=A_\zeta^{-1}\bigl(y-p_*+A_\zeta(f(\zeta))\bigr),
\]
and is therefore jointly holomorphic. Define the holomorphic invertible matrix
\begin{equation}\label{eq:normalizing-matrix}
M_\zeta=D\widetilde A_\zeta(f(\zeta)).
\end{equation}

\smallskip
\noindent\textbf{Step 3: Construction of the normalized family, its total inverse, and the basin dynamics.}
Set
\begin{equation}\label{eq:general-final-F}
F(\zeta,\xi)=\widetilde A_\zeta^{-1}\bigl(\Phi(M_\zeta\xi)\bigr),
\qquad (\zeta,\xi)\in V\times\C^n.
\end{equation}
This map is jointly holomorphic. Its values avoid \(E\), since otherwise \(\Phi(M_\zeta\xi)\) would belong to \(\widetilde A_\zeta(E)\subset T_S\), contrary to \eqref{eq:fixed-model}. Also,
\[
F(\zeta,0)=\widetilde A_\zeta^{-1}(p_*)=f(\zeta),
\qquad
D_\xi F(\zeta,0)=M_\zeta^{-1}D\Phi(0)M_\zeta=I_n.
\]
Each fiber is a biholomorphism onto \(\Omega_\zeta=\widetilde A_\zeta^{-1}(\Omega)\). The total image
\[
\mathcal U=\{(\zeta,z)\in V\times\C^n:\widetilde A_\zeta(z)\in\Omega\}
\]
is open. On \(\mathcal U\), the total inverse is
\begin{equation}\label{eq:general-total-inverse}
(\zeta,z)\longmapsto
\bigl(\zeta,M_\zeta^{-1}\Phi^{-1}(\widetilde A_\zeta(z))\bigr),
\end{equation}
which is jointly holomorphic. This proves the total biholomorphism assertion.

Finally, define
\begin{equation}\label{eq:conjugated-dynamics}
g_\zeta=\widetilde A_\zeta^{-1}\circ\Psi\circ \widetilde A_\zeta.
\end{equation}
This is a holomorphic family of automorphisms, it fixes \(f(\zeta)\), and its derivative there is conjugate to \(D\Psi(p_*)\). Hence the fixed point is attracting. For every integer \(k\geq0\),
\[
g_\zeta^k=\widetilde A_\zeta^{-1}\circ\Psi^k\circ \widetilde A_\zeta.
\]
It follows, using continuity of \(\widetilde A_\zeta\) and \(\widetilde A_\zeta^{-1}\), that the attracting basin of \(f(\zeta)\) is exactly \(\widetilde A_\zeta^{-1}(\Omega)=\Omega_\zeta\). Since \(E\neq\varnothing\), these domains are proper. All the assertions follow.
\end{proof}

The normalized families also give dominating sprays on the quotient complement.

\begin{corollary}\label{cor:quotient-complement}
For every \(n\geq2\) and every compact \(\OO(X)\)-convex set \(C\subset X=(\C^*)^n\), the complement \(X\setminus C\) is Oka.
\end{corollary}

\begin{proof}
If \(C=\varnothing\), every holomorphic map \(g=(g_1,\ldots,g_n):U\to X\) admits the dominating holomorphic spray
\[
S(\zeta,\xi)=\bigl(g_1(\zeta)e^{\xi_1},\ldots,g_n(\zeta)e^{\xi_n}\bigr),
\qquad (\zeta,\xi)\in U\times\C^n.
\]
Theorem~\ref{thm:kusakabe} therefore shows that \(X\) is Oka.

Suppose \(C\neq\varnothing\). Let \(m\geq1\), \(K\subset\C^m\) be compact and convex, \(U_0\subset\C^m\) be an open neighborhood of \(K\), and \(g:U_0\to X\setminus C\) be holomorphic. We may assume \(K\neq\varnothing\). Choose a convex open neighborhood \(U\subset U_0\) of \(K\). The nonvanishing coordinate functions \(g_j\) have holomorphic logarithms \(\ell_j\) on \(U\), so
\[
f=\frac1{\ii}(\ell_1,\ldots,\ell_n):U\longrightarrow\C^n\setminus q^{-1}(C)
\qquad\text{satisfies}\qquad q\circ f=g|_U.
\]
Theorem~\ref{thm:periodic-complement} gives a neighborhood \(V\subset U\) of \(K\) and a holomorphic map \(F:V\times\C^n\to\C^n\setminus q^{-1}(C)\) with \(F(\zeta,0)=f(\zeta)\) and \(D_\xi F(\zeta,0)=I_n\). Hence
\[
S=q\circ F:V\times\C^n\longrightarrow X\setminus C,
\qquad S(\zeta,0)=g(\zeta),
\]
and
\[
D_\xi S(\zeta,0)=Dq(f(\zeta))D_\xi F(\zeta,0)=Dq(f(\zeta)).
\]
This derivative is invertible because \(q\) is locally biholomorphic. Theorem~\ref{thm:kusakabe} now gives the Oka property of \(X\setminus C\).
\end{proof}

\subsection{Invariant tube families}\label{subsec:invariant-tube-families}

The fixed-model construction preserves the set \(\widetilde A_\zeta^{-1}(T_S)\supset E\). For positive-width rectangular tubes, the following refinement makes the original obstacle strictly forward invariant.

\begin{proposition}\label{prop:family}
Fix \(n\geq2\) and \(\delta>0\). Let \(m\geq1\), \(K\subset\C^m\) be compact and convex, \(U_0\subset\C^m\) be an open neighborhood of \(K\), and \(f:U_0\to\C^n\setminus T_\delta\) be holomorphic. There exist an open neighborhood \(V\subset U_0\) of \(K\), a holomorphic map \(F:V\times\C^n\to\C^n\setminus T_\delta\), and a holomorphic family \(\Psi_\zeta\in\Aut(\C^n)\) with all the family and basin properties in Theorem~\ref{thm:periodic-complement} for \(E=T_\delta\), such that \(\Omega_\zeta=F_\zeta(\C^n)\) is the attracting basin of \(\Psi_\zeta\) at \(f(\zeta)\) and
\[
\Psi_\zeta(T_\delta)\subset\Int T_\delta\qquad(\zeta\in V).
\]
\end{proposition}

The strict inclusion requires \(\delta>0\), since \(\Int T_0=\varnothing\).

For the proof of Proposition~\ref{prop:family}, we first establish parameter-uniform basin coordinates. The following proposition gives a uniform parameter version of the classical linearization argument of Rosay and Rudin \cite[Theorem~9.1]{RR88}. Related constructions for nonautonomous basins appear in \cite[Theorem~4]{Wol05}, and their use in holomorphic families is described in \cite[proof of Theorem~2.3]{FW24}. We include the proof to establish joint holomorphic dependence under a uniform pinching condition.

\begin{proposition}\label{prop:basins}
Let \(V\) be a complex manifold, and \(g:V\times\C^d\to\C^d\) be holomorphic, with \(g_\zeta\coloneqq g(\zeta,\cdot)\in\Aut(\C^d)\) and \(g_\zeta(0)=0\). Set \(B_\zeta\coloneqq Dg_\zeta(0)\). Suppose there exist constants
\begin{equation}\label{eq:pinching}
R>0,\quad C\geq0,\quad 0<a\leq b<1,\quad b^2<a,
\end{equation}
such that for every \(\zeta\in V\) and \(|u|<R\),
\begin{align}
|g_\zeta(u)|&\leq b|u|,\label{eq:basin-contract}\\
\norm{B_\zeta^{-1}}&\leq a^{-1},\qquad \norm{B_\zeta}\leq b,\label{eq:basin-linear-bounds}\\
|g_\zeta(u)-B_\zeta u|&\leq C|u|^2. \label{eq:basin-quadratic}
\end{align}
Let
\[
\mathcal B_\zeta\coloneqq \{u\in\C^d:g_\zeta^n(u)\longrightarrow0\}.
\]
Then the total basin
\[
\mathcal B\coloneqq \{(\zeta,u):\zeta\in V,\ u\in\mathcal B_\zeta\}
\]
is open, and there is a biholomorphism
\begin{equation}\label{eq:total-basin-biholo}
\mathcal H:\mathcal B\longrightarrow V\times\C^d, \qquad \mathcal H(\zeta,u)=(\zeta,H_\zeta(u)),
\end{equation}
satisfying
\begin{equation}\label{eq:basin-normalization}
H_\zeta(0)=0,\qquad D H_\zeta(0)=I_d,\qquad H_\zeta\circ g_\zeta=B_\zeta H_\zeta.
\end{equation}
In particular, its fiberwise inverse depends holomorphically on both variables and is defined for every point of \(V\times\C^d\).
\end{proposition}

\begin{proof}
We begin with a local linearization on a common ball. By \eqref{eq:basin-contract}, every iterate of a point of \(\Ball_R\) stays in that ball. Set
\begin{equation}\label{eq:h-n}
h_{\zeta,n}(u)\coloneqq B_\zeta^{-n}g_\zeta^n(u), \qquad |u|<R,\quad n\geq0.
\end{equation}
Each function is jointly holomorphic in \((\zeta,u)\); matrix inversion is holomorphic because \(B_\zeta\) is invertible. With \(R_\zeta(u)\coloneqq g_\zeta(u)-B_\zeta u\), we obtain
\begin{align}
h_{\zeta,n+1}(u)-h_{\zeta,n}(u) &=B_\zeta^{-(n+1)}R_\zeta(g_\zeta^n(u)),\label{eq:h-telescope}\\
|h_{\zeta,n+1}(u)-h_{\zeta,n}(u)| &\leq\frac Ca\left(\frac{b^2}{a}\right)^n|u|^2. \label{eq:h-geometric}
\end{align}
The ratio \(b^2/a\) is strictly less than one. The sequence therefore converges uniformly on compact subsets of \(V\times\Ball_R\) to a jointly holomorphic map \(h_\zeta(u)\), and
\begin{equation}\label{eq:h-quadratic}
|h_\zeta(u)-u|\leq D|u|^2, \qquad D\coloneqq \frac{C}{a(1-b^2/a)}.
\end{equation}
Moreover,
\begin{equation}\label{eq:h-normalization}
h_\zeta(0)=0,\qquad Dh_\zeta(0)=I_d.
\end{equation}
The second equality follows either from \eqref{eq:h-quadratic} or from differentiating the locally uniform limit. The identity
\[
h_{\zeta,n}(g_\zeta(u))=B_\zeta h_{\zeta,n+1}(u)
\]
passes to the limit and gives
\begin{equation}\label{eq:h-conjugacy}
h_\zeta(g_\zeta(u))=B_\zeta h_\zeta(u),\qquad |u|<R.
\end{equation}
Choose \(r_0>0\) with \(2r_0<R\), small enough that \(4Dr_0<1/2\) when \(D>0\). For \(|u|\leq r_0\) and a unit vector \(v\), applying Cauchy's formula on the complex line \(u+\lambda v\), \(|\lambda|=r_0\), to \eqref{eq:h-quadratic} gives
\[
|D(h_\zeta-\id)(u)v|\leq4Dr_0<1/2.
\]
If \(D=0\), the same conclusion is immediate. Hence, after decreasing \(r_0\) if necessary,
\begin{equation}\label{eq:h-derivative-small}
\norm{Dh_\zeta(u)-I_d}\leq1/2 \quad(\zeta\in V,\ |u|\leq r_0).
\end{equation}
Integrating along line segments in the convex ball yields
\begin{equation}\label{eq:h-bilipschitz}
|h_\zeta(u)-h_\zeta(v)|\geq\tfrac12|u-v|, \qquad u,v\in\Ball_{r_0}.
\end{equation}
Thus \(h_\zeta\) is injective there and has everywhere invertible derivative.

It also satisfies
\begin{equation}\label{eq:h-image-ball}
\Ball_{r_0/2}\subset h_\zeta(\Ball_{r_0}) \qquad(\zeta\in V).
\end{equation}
Indeed, for \(|w|<r_0/2\), the map
\[
u\longmapsto w-(h_\zeta(u)-u)
\]
is a contraction of the closed ball \(\overline\Ball_{r_0}\) into itself, by \eqref{eq:h-derivative-small}, and its fixed point satisfies \(|u|\leq2|w|<r_0\). The inverse
\begin{equation}\label{eq:local-inverse}
k(\zeta,w)\coloneqq h_\zeta^{-1}(w),\qquad |w|<r_0/2,
\end{equation}
is jointly holomorphic: the holomorphic inverse function theorem applies locally to \((\zeta,u)\mapsto(\zeta,h_\zeta(u))\), and uniqueness makes the local inverses agree.

The local coordinate extends to the whole basin as follows. Since \(g_\zeta(\Ball_{r_0})\subset\Ball_{r_0}\), one has
\begin{equation}\label{eq:basin-union}
\mathcal B=\bigcup_{n\geq0} \{(\zeta,u):g_\zeta^n(u)\in\Ball_{r_0}\}.
\end{equation}
The sets on the right are open and increasing, so \(\mathcal B\) is open. Equality holds because entering the ball forces geometric convergence to zero, whereas every orbit converging to zero eventually enters it.

For \(u\in\mathcal B_\zeta\), choose \(n\) so large that \(g_\zeta^n(u)\in\Ball_{r_0}\), and set
\begin{equation}\label{eq:H-global}
H_\zeta(u)\coloneqq B_\zeta^{-n}h_\zeta(g_\zeta^n(u)).
\end{equation}
Equation~\eqref{eq:h-conjugacy} shows that \eqref{eq:H-global} is independent of \(n\) once the orbit has entered \(\Ball_{r_0}\). These formulas agree on the overlaps of the open sets in \eqref{eq:basin-union}, so \(H(\zeta,u)\coloneqq H_\zeta(u)\) is jointly holomorphic on the total basin. Formula \eqref{eq:H-global} also shows that the fiber derivative of \(H_\zeta\) is invertible everywhere and gives the identities in \eqref{eq:basin-normalization}.

It remains to prove global bijectivity and holomorphic dependence of the inverse. Suppose \(H_\zeta(u)=H_\zeta(v)\). Choose one integer \(n\) taking both points into \(\Ball_{r_0}\). Equation~\eqref{eq:H-global}, invertibility of \(B_\zeta\), and injectivity of \(h_\zeta\) on the small ball imply \(g_\zeta^n(u)=g_\zeta^n(v)\). Since \(g_\zeta^n\) is an automorphism, \(u=v\).

Given \(w\in\C^d\), choose \(n\) so large that \(b^n|w|<r_0/2\). By \eqref{eq:basin-linear-bounds} and \eqref{eq:h-image-ball}, \(B_\zeta^n w\in h_\zeta(\Ball_{r_0})\). Thus
\begin{equation}\label{eq:global-inverse-formula}
u\coloneqq g_\zeta^{-n}\bigl(k(\zeta,B_\zeta^n w)\bigr)
\end{equation}
belongs to \(\mathcal B_\zeta\) and satisfies \(H_\zeta(u)=w\). This proves surjectivity.

Thus the total map \eqref{eq:total-basin-biholo} is bijective. Its derivative is block triangular with identity in the parameter block and invertible derivative \(D_uH_\zeta\) in the fiber block. The holomorphic inverse function theorem therefore gives a globally holomorphic inverse.

\end{proof}

\begin{proof}[Proof of Proposition~\ref{prop:family}]
First choose a compact convex set \(L\) and a bounded convex open set \(U\) so that
\begin{equation}\label{eq:parameter-nesting}
K\subset\Int L\subset L\Subset U\Subset U_0.
\end{equation}
For example, sufficiently small Euclidean thickenings of \(K\) give these sets. Since \(f(L)\) is compact in \(\C^n\setminus T_\delta\), the prescribed width satisfies
\[
\min_{\zeta\in L}\norm{\im f(\zeta)}_\infty>\delta.
\]
Set \(V\coloneqq \Int L\), an open neighborhood of \(K\).

Set
\[
C_\delta=\{w\in X:e^{-\delta}\leq|w_j|\leq e^\delta,\ j=1,\ldots,n\}.
\]
This compact set is \(\OO(X)\)-convex: the functions \(w_j\) and \(w_j^{-1}\) separate points violating the upper and lower bounds, respectively. The holomorphic path \(p(s,\zeta)=(1+s)f(\zeta)\) satisfies
\[
\norm{\im p(t,\zeta)}_\infty
=(1+t)\norm{\im f(\zeta)}_\infty>\delta
\qquad(t\in[0,1],\ \zeta\in L),
\]
so it avoids \(q^{-1}(C_\delta)=T_\delta\) and Lemma~\ref{lem:transport} applies. Fix the radius \(r>0\) supplied by that lemma, independently of the approximation tolerance, and choose
\begin{equation}\label{eq:final-epsilon}
0<\epsilon<\min\{\delta/4,1/10\}.
\end{equation}
For this tolerance, the lemma gives a holomorphic family of \(\Gamma\)-equivariant automorphisms \(A_\zeta\in\Aut(\C^n)\), \(\zeta\in U\), satisfying
\begin{align}
\sup_{\zeta\in L,\ z\in T_\delta}|A_\zeta(z)-z|&<\epsilon,\label{eq:A-tube}\\
\sup_{\zeta\in L,\ |u|\leq r}|A_\zeta(f(\zeta)+u)-(2f(\zeta)+u)|&<\epsilon,\label{eq:A-translation}\\
\sup_{\zeta\in L,\ |u|\leq r}\norm{D_zA_\zeta(f(\zeta)+u)-I_n}&<\epsilon.\label{eq:A-derivative}
\end{align}
The inverse family is holomorphic and \(\det D_zA_\zeta\equiv1\). Set
\begin{align}
c(\zeta)&\coloneqq f(\zeta)-\tfrac12A_\zeta(f(\zeta)), \label{eq:correction}\\
\Psi_\zeta(z)&\coloneqq \tfrac12 A_\zeta(z)+c(\zeta). \label{eq:Psi}
\end{align}
The maps \(c\) and \(\Psi\) are holomorphic, each \(\Psi_\zeta\) is an automorphism, and the correction fixes the prescribed section:
\begin{equation}\label{eq:Psi-fixed}
\Psi_\zeta(f(\zeta))=f(\zeta).
\end{equation}
At \(u=0\), \eqref{eq:A-translation} gives
\begin{equation}\label{eq:c-small}
|c(\zeta)|<\epsilon/2\qquad(\zeta\in L).
\end{equation}

The real tube is forward invariant. For \(\zeta\in L\) and \(z\in T_\delta\), equations \eqref{eq:A-tube}, \eqref{eq:Psi}, and \eqref{eq:c-small} give
\begin{equation}\label{eq:invariant-tube}
\norm{\im\Psi_\zeta(z)}_\infty\leq\tfrac12\norm{\im z}_\infty+\tfrac12|A_\zeta(z)-z|+|c(\zeta)|<\tfrac12\delta+\epsilon<\delta.
\end{equation}
Hence
\begin{equation}\label{eq:tube-forward-invariant}
\Psi_\zeta(T_\delta)\subset\Int T_\delta.
\end{equation}

The prescribed section is uniformly attracting. In coordinates centered at the fixed section, we write
\begin{equation}\label{eq:recenter-g}
g_\zeta(u)\coloneqq \Psi_\zeta(f(\zeta)+u)-f(\zeta), \qquad B_\zeta\coloneqq Dg_\zeta(0).
\end{equation}
Then \(g_\zeta(0)=0\), and, by \eqref{eq:A-derivative} and \eqref{eq:final-epsilon},
\begin{equation}\label{eq:derivative-pinched}
\norm{Dg_\zeta(u)-\tfrac12 I_n} <\epsilon/2<1/20, \qquad\zeta\in L,\quad |u|\leq r.
\end{equation}
Integration along the segment from \(0\) to \(u\) implies
\begin{equation}\label{eq:g-near-half}
|g_\zeta(u)-\tfrac12u|\leq\tfrac1{20}|u|.
\end{equation}
We may therefore take
\begin{equation}\label{eq:ab-explicit}
a=\frac25,\qquad b=\frac35,\qquad \frac{b^2}{a}=\frac9{10}<1.
\end{equation}
Indeed, \(|g_\zeta(u)|\leq(11/20)|u|\leq b|u|\), while the smallest singular value of \(B_\zeta\) is at least \(9/20>a\). Thus
\[
\norm{B_\zeta^{-1}}\leq20/9<a^{-1}, \qquad\norm{B_\zeta}\leq11/20<b.
\]

Take \(R=r/2\). The second fiber derivatives of \(g\) are bounded on the compact set \(L\times\overline\Ball_R\), since \(g\) is holomorphic on \(U\times\C^n\). Taylor's formula with integral remainder therefore gives a constant \(C\), independent of \(\zeta\), such that
\[
|g_\zeta(u)-B_\zeta u|\leq C|u|^2 \qquad(\zeta\in V,\ |u|<R).
\]
All hypotheses of Proposition~\ref{prop:basins} are now satisfied with constants independent of \(\zeta\).

To pass to the attracting basins, let
\begin{equation}\label{eq:Omega-Psi}
\Omega_\zeta\coloneqq \{z\in\C^n:\Psi_\zeta^\nu(z)\longrightarrow f(\zeta)\text{ as }\nu\to\infty\}.
\end{equation}
By construction \(f(\zeta)\notin T_\delta\). If \(z\in T_\delta\), all iterates \(\Psi_\zeta^\nu(z)\) remain in the closed set \(T_\delta\), by \eqref{eq:tube-forward-invariant}. They cannot converge to \(f(\zeta)\). Therefore
\begin{equation}\label{eq:Omega-avoid}
\Omega_\zeta\cap T_\delta=\varnothing.
\end{equation}

Let \(H_\zeta:\mathcal B_\zeta\to\C^n\) be the basin coordinate from Proposition~\ref{prop:basins}, where \(\mathcal B_\zeta=\Omega_\zeta-f(\zeta)\), and write \(G_\zeta:\C^n\to\mathcal B_\zeta\) for its inverse. Set
\begin{equation}\label{eq:final-F}
F(\zeta,\xi)\coloneqq f(\zeta)+G_\zeta(\xi).
\end{equation}
The map is jointly holomorphic on \(V\times\C^n\); its image in each fiber is \(\Omega_\zeta\subset\C^n\setminus T_\delta\). The normalization in Proposition~\ref{prop:basins} gives
\[
F(\zeta,0)=f(\zeta),\qquad D_\xi F(\zeta,0)=I_n.
\]
The total biholomorphism follows from Proposition~\ref{prop:basins} after the holomorphic change of fiber coordinate \(z=f(\zeta)+u\). This proves Proposition~\ref{prop:family}.
\end{proof}

\section{Convex tubes and Cartesian-product complements}\label{sec:applications}

\subsection{Tube complements and totally real planes}\label{subsec:real-plane}

For a nonempty compact convex set \(B\subset\R^n\), let
\begin{equation}\label{eq:logarithmic-tube-quotient}
C_B=\bigl\{w\in X:(\log|w_1|,\ldots,\log|w_n|)\in-B\bigr\}.
\end{equation}
Then \(q^{-1}(C_B)=T_B\). The following elementary separation argument identifies the convexity needed in Theorem~\ref{thm:periodic-complement}.

\begin{lemma}\label{lem:convex-tube-quotient}
The set \(C_B\) is compact and \(\OO(X)\)-convex.
\end{lemma}

\begin{proof}
The set \(C_B\) is closed in \(X\), and compactness of \(-B\) bounds each coordinate modulus above and away from zero, so \(C_B\) is compact.

For \(w\notin C_B\), strict real linear separation of \((\log|w_1|,\ldots,\log|w_n|)\) from \(-B\) gives \(a\in\R^n\) such that the globally pluriharmonic function
\[
u_a(v)=\sum_{j=1}^n a_j\log|v_j|,\qquad v\in X,
\]
satisfies \(u_a(w)>\sup_{C_B}u_a\). Hence the hull of \(C_B\) with respect to global smooth plurisubharmonic functions is \(C_B\). On the Stein manifold \(X\), this hull equals the holomorphic hull; see \cite[\S2, p.~338]{For22SN}. Thus \(C_B\) is \(\OO(X)\)-convex.
\end{proof}

\begin{proof}[Proof of Theorem~\ref{thm:main}\,(i)]
The case \(B=\varnothing\) gives \(\C^n\), so assume \(B\neq\varnothing\). By Lemma~\ref{lem:convex-tube-quotient}, Theorem~\ref{thm:periodic-complement} applies with \(C=C_B\). It proves the Oka property and, in addition, provides normalized holomorphic families of attracting basins over every prescribed holomorphic section. This includes lower-dimensional cross-sections and \(B=\{0\}\); no interior or boundary regularity assumption has been used.
\end{proof}

\begin{corollary}\label{cor:affine}
For \(n\geq2\), let \(P=p+L\subset\C^n\) be a maximally totally real affine plane and \(D\subset L\) be compact and convex, where \(L\) is a real vector space of dimension \(n\). Then \(\C^n\setminus(P+\ii D)\) is Oka. In particular, \(\C^n\setminus P\) is Oka.
\end{corollary}

\begin{proof}
Choose a real basis \(v_1,\ldots,v_n\) of \(L\). These vectors are complex linearly independent: if \(\sum_j(a_j+\ii b_j)v_j=0\) with real \(a_j,b_j\), then \(\sum_j a_jv_j=-\ii\sum_j b_jv_j\) belongs to \(L\cap\ii L=\{0\}\). Real linear independence gives \(a_j=b_j=0\) for every \(j\). Thus \(A(z)=\sum_j z_jv_j\) is a complex linear automorphism. The set \(B=A^{-1}(D)\subset\R^n\) is compact and convex, and the affine complex automorphism \(z\mapsto p+A(z)\) carries \(T_B\) onto \(P+\ii D\). The conclusion follows from Theorem~\ref{thm:main}\,(i) and biholomorphic invariance of the Oka property. Take \(D=\{0\}\) for the last assertion.
\end{proof}

\subsection{Products of planar compact sets}

We first characterize the punctures for which a planar compact set becomes holomorphically convex, and then use logarithmic covers and localization to prove the general product result.

\begin{lemma}\label{lem:punctured-convexity}
Let \(K\subset\C\) be a nonempty compact set and \(a\in\C\setminus K\). The set \(K-a\subset\C^*\) is \(\OO(\C^*)\)-convex if and only if every bounded connected component of \(\C\setminus K\) contains \(a\).
\end{lemma}

\begin{proof}
Translation identifies \(\C\setminus\{a\}\) with \(\C^*\), so the assertion is equivalent to the corresponding holomorphic convexity of \(K\) in the punctured plane. By the Runge criterion for open Riemann surfaces \cite[\S25]{Forster81}, this holds exactly when \((\C\setminus\{a\})\setminus K\) has no relatively compact connected component in \(\C\setminus\{a\}\).

Let \(D_a\) be the component of \(\C\setminus K\) containing \(a\). Every other component remains a component after puncturing. The set \(D_a\setminus\{a\}\) is also connected: an open connected subset of \(\C\) is polygonally connected, and any polygonal path meeting \(a\) can be modified inside a small disk about \(a\), contained in \(D_a\), to avoid its center. Thus the components after puncturing are precisely \(D_a\setminus\{a\}\) and the original components distinct from \(D_a\).

If a bounded component \(D\neq D_a\) occurs, its closure is compact in \(\C\) and avoids \(a\). Indeed, a disk about \(a\) contained in \(D_a\) is disjoint from \(D\) and hence separates \(a\) from \(\overline D\). Thus \(D\) is relatively compact in the punctured plane. Conversely, every unbounded component has noncompact closure in the punctured plane, and \(D_a\setminus\{a\}\) is not relatively compact there, since its closure in \(\C\) contains the missing point \(a\). These observations show that a relatively compact component after puncturing exists exactly when a bounded component of \(\C\setminus K\) does not contain \(a\). The Runge criterion gives the equivalence.
\end{proof}

The following proposition applies this convexity condition on finitely many shifted logarithmic covers.

\begin{proposition}\label{prop:product-localization}
For \(n\geq2\), let \(K_1,\ldots,K_n\subset\C\) be nonempty compact sets. Suppose that, for every \(j\), there are two distinct points \(a_j^0,a_j^1\in\C\setminus K_j\) such that
\[
K_j-a_j^\nu\subset\C^*\quad\text{is }\OO(\C^*)\text{-convex},\qquad \nu\in\{0,1\}.
\]
Then \(M=\C^n\setminus(K_1\times\cdots\times K_n)\) is Oka.
\end{proposition}

\begin{proof}
Fix a choice \(a=(a_1,\ldots,a_n)\), where \(a_j\in\{a_j^0,a_j^1\}\), and define
\begin{align*}
C_{j,a_j}&=K_j-a_j\subset\C^*,
& E_{j,a_j}&=\{u\in\C:a_j+e^{\ii u}\in K_j\},\\
C_a&=\prod_{j=1}^n C_{j,a_j},
& E_a&=\prod_{j=1}^n E_{j,a_j}=q^{-1}(C_a),\\
&&Y_a&=\C^n\setminus E_a.
\end{align*}

\smallskip
\noindent\textbf{Step 1: Holomorphic convexity on the quotient.}
The set \(C_a\) is nonempty and compact in \(X\): each factor \(K_j-a_j\) is nonempty and compact, and lies in \(\C^*\) because \(a_j\notin K_j\). It is also \(\OO(X)\)-convex. Indeed, if \(w\notin C_a\), some coordinate \(w_j\) does not belong to \(C_{j,a_j}\). A holomorphic function on \(\C^*\) separating \(w_j\) from that factor, viewed as a function of the \(j\)-th coordinate on \(X\), separates \(w\) from \(C_a\).
By Lemma~\ref{lem:periodic-connectedness}, \(Y_a\) is path connected. Theorem~\ref{thm:periodic-complement} applies to \(C_a\): the manifold \(Y_a\) is Oka and has the normalized families asserted there.

\smallskip
\noindent\textbf{Step 2: Descent of the sprays to a Zariski open subset.}
Define
\begin{equation}\label{eq:general-product-patch}
M_a=M\setminus\bigcup_{j=1}^n\{x_j=a_j\}.
\end{equation}
The shifted exponential map
\begin{equation}\label{eq:shifted-product-cover}
\pi_a:Y_a\longrightarrow M_a,
\qquad \pi_a(z)=(a_1+e^{\ii z_1},\ldots,a_n+e^{\ii z_n}),
\end{equation}
is a holomorphic covering: it is the restriction of the product exponential covering onto \(\prod_j(\C\setminus\{a_j\})\), and the inverse image of \(K_1\times\cdots\times K_n\) under that covering is exactly \(E_a\).

Let \(K\subset\C^m\) be compact and convex and \(g\) be holomorphic from a neighborhood of \(K\) to \(M_a\). Shrink that neighborhood to a convex open set \(U\) containing \(K\). The nonvanishing holomorphic functions \(g_j-a_j\) admit holomorphic logarithms \(\ell_j\) on the simply connected set \(U\). Thus
\[
f=(\ell_1/\ii,\ldots,\ell_n/\ii):U\longrightarrow Y_a
\]
is a holomorphic lift of \(g\). Theorem~\ref{thm:periodic-complement} gives a neighborhood \(V\subset U\) of \(K\) and a holomorphic map \(F:V\times\C^n\to Y_a\) with \(F(\zeta,0)=f(\zeta)\) and \(D_\xi F(\zeta,0)=I_n\). Then
\[
G=\pi_a\circ F:V\times\C^n\longrightarrow M_a
\]
satisfies \(G(\zeta,0)=g(\zeta)\), and
\[
D_\xi G(\zeta,0)=D\pi_a(f(\zeta))
=\operatorname{diag}\bigl(\ii e^{\ii f_1(\zeta)},\ldots,\ii e^{\ii f_n(\zeta)}\bigr)
\]
is invertible. Theorem~\ref{thm:kusakabe} proves that \(M_a\) is Oka.

\smallskip
\noindent\textbf{Step 3: Construction of the finite Zariski localization cover.}
The subset \(M_a\) is Zariski open in \(M\), since its complement is the zero set in \(M\) of the holomorphic polynomial \(\prod_{j=1}^n(x_j-a_j)\). Moreover,
\begin{equation}\label{eq:finite-product-cover}
M=\bigcup_{a\in\prod_{j=1}^n\{a_j^0,a_j^1\}} M_a.
\end{equation}
To verify the cover, fix \(x\in M\). Since \(a_j^0\neq a_j^1\), for each \(j\) at least one of these two numbers differs from \(x_j\). Choosing such an \(a_j\) for every \(j\) gives \(x\in M_a\). There are \(2^n\) choices of \(a\).
Each \(M_a\) is connected as the continuous image of the path-connected set \(Y_a\). They also have a common point: choose every coordinate sufficiently large in modulus to be outside its compact factor and different from the finitely many chosen centers. Hence their union \(M\) is connected. Applying Theorem~\ref{thm:localization} to \eqref{eq:finite-product-cover} proves that \(M\) is Oka.
\end{proof}

\begin{proof}[Proof of Theorem~\ref{thm:main}\,(ii)]
If one factor is empty, the product is empty and the complement is \(\C^n\), which is Oka. Assume that all factors are nonempty. For each \(j\), make the following choice. If \(\C\setminus K_j\) has no bounded component, choose any two distinct points \(a_j^0,a_j^1\in\C\setminus K_j\). If it has one bounded component \(D_j\), choose two distinct points \(a_j^0,a_j^1\in D_j\); these exist because \(D_j\) is nonempty and open. In either case, each chosen point belongs to every bounded component of \(\C\setminus K_j\), with this condition vacuous in the first case. Lemma~\ref{lem:punctured-convexity} shows that \(K_j-a_j^\nu\) is \(\OO(\C^*)\)-convex for \(\nu=0,1\). All the hypotheses of Proposition~\ref{prop:product-localization} are satisfied, and it proves assertion~(ii).
\end{proof}

If a factor has two distinct bounded complementary components, no single puncture belongs to both, so Lemma~\ref{lem:punctured-convexity} shows that the present localization construction does not apply.

For \(n\geq2\) and \(0<r_j\leq R_j<\infty\), \(j=1,\ldots,n\), consider the closed annuli
\[
K_j=\{x\in\C:r_j\leq|x|\leq R_j\}.
\]
The complement of \(K_j\) consists of the disk \(\{|x|<r_j\}\) and the exterior \(\{|x|>R_j\}\), exactly one of which is bounded. This remains true when \(r_j=R_j\). Theorem~\ref{thm:main}\,(ii) therefore gives the Oka property of \(\C^n\setminus A_{\boldsymbol r,\boldsymbol R}\) for all the indicated radii. Taking \(r_j=R_j=1\) for every \(j\) gives the Oka property of \(\C^n\setminus T\).

For \(n\geq2\), arbitrary centers \(c_1,\ldots,c_n\in\C\), and positive radii \(r_1,\ldots,r_n\), the complement of
\[
\{x\in\C^n:|x_j-c_j|=r_j\text{ for }j=1,\ldots,n\}
\]
is Oka. Indeed, the affine automorphism
\[
(x_1,\ldots,x_n)\longmapsto
\left(\frac{x_1-c_1}{r_1},\ldots,\frac{x_n-c_n}{r_n}\right)
\]
maps this product of circles onto \(T\).

\section*{Acknowledgments}

We thank Professor Franc Forstneri\v c for bringing to our attention the question of whether the product of the unit circles in the coordinate planes of \(\C^2\) has an Oka complement. He also observed that our method yields the Oka property of \(\C^2\setminus T_\delta\) for every \(\delta\geq0\).

\section*{Funding}

S.-Y.~Xie was partially supported by the National Key R\&D Program of China (grant nos.~2023YFA1010500 and 2021YFA1003100), the National Natural Science Foundation of China (grant nos.~12288201 and 12471081), and the Xiaomi Young Talents Program.

\section*{AI use statement}

The mathematical insights and ideas are due to the authors. AI tools assisted in drafting the manuscript. The authors independently checked the entire manuscript and take full responsibility for its content.

\end{document}